\documentclass[11pt,a4paper]{article}

\usepackage{inputenc}
\usepackage{amsmath}
\usepackage{bm}
\usepackage{bbold}
\usepackage{amsthm}
\usepackage{enumerate}

\usepackage{pseudocode}

\usepackage{hyperref}
\usepackage{breakurl}

\title{Complete algebraic solution of multidimensional optimization problems in tropical semifield}

\author{N. Krivulin\thanks{Faculty of Mathematics and Mechanics, Saint Petersburg State University, 28 Universitetsky Ave., St.~Petersburg, 198504, Russia, 
nkk@math.spbu.ru.}
}

\date{}

\newtheorem{theorem}{Theorem}
\newtheorem{lemma}[theorem]{Lemma}
\newtheorem{corollary}[theorem]{Corollary}

\theoremstyle{definition}
\newtheorem{example}{Example}

\begin{document}

\maketitle

\begin{abstract}
We consider multidimensional optimization problems that are formulated in the framework of tropical mathematics to minimize functions defined on vectors over a tropical semifield (a semiring with idempotent addition and invertible multiplication). The functions, given by a matrix and calculated through multiplicative conjugate transposition, are nonlinear in the tropical mathematics sense. We start with known results on the solution of the problems with irreducible matrices. To solve the problems in the case of arbitrary (reducible) matrices, we first derive the minimum value of the objective function, and find a set of solutions. We show that all solutions of the problem satisfy a system of vector inequalities, and then use these inequalities to establish characteristic properties of the solution set. Furthermore, all solutions of the problem are represented as a family of subsets, each defined by a matrix that is obtained by using a matrix sparsification technique. We describe a backtracking procedure that allows one to reduce the brute-force generation of sparsified matrices by skipping those, which cannot provide solutions, and thus offers an economical way to obtain all subsets in the family. Finally, the characteristic properties of the solution set are used to provide complete solutions in a closed form. We illustrate the results obtained with simple numerical examples.
\\

\\

\end{abstract}

\section{Introduction}

Tropical (idempotent) mathematics, which deals with the theory and applications of semirings with idempotent addition, dates back to a few seminal works \cite{Pandit1961Anew,Cuninghamegreen1962Describing,Giffler1963Scheduling,Hoffman1963Onabstract,Vorobjev1963Theextremal,Romanovskii1964Asymptotic} which appeared in the early 1960s. Today, tropical mathematics is a rapidly evolving area (see, e.g., recent publications \cite{Kolokoltsov1997Idempotent,Golan2003Semirings,Heidergott2006Maxplus,Itenberg2007Tropical,Gondran2008Graphs,Butkovic2010Maxlinear,Mceneaney2010Maxplus,Maclagan2015Introduction}), which offers a useful analytical and computational framework to solve many recent problems in operations research, computer science and other fields. These problems can be formulated and solved as optimization problems in the tropical mathematics setting, hence are referred to as tropical optimization problems. Typical examples of the application areas of tropical optimization include project scheduling \cite{Zimmermann2003Disjunctive,Tam2010Optimizing,Aminu2012Nonlinear,Krivulin2015Extremal,Krivulin2015Amultidimensional,Krivulin2017Direct}, location analysis \cite{Hudec1999Biobjective,Tharwat2010Oneclass,Krivulin2014Complete}, and decision making \cite{Elsner2004Maxalgebra,Elsner2010Maxalgebra,Gavalec2015Decision,Krivulin2016Using}.

Many tropical optimization problems are formulated to minimize or maximize functions defined on vectors over idempotent semifields (semirings with multiplicative inverses). These problems may have functions to optimize (objective functions), which can be linear or non-linear in the tropical mathematics sense, and constraints, which can take the form of vector inequalities and equalities. Some problems have direct, explicit solutions obtained using general assumptions. For other problems, only algorithmic solutions under restrictive conditions are known, which apply iterative numerical procedures to find a solution if it exists, or to indicate infeasibility of the problem otherwise. A short overview of tropical optimization problems and their solutions can be found in \cite{Krivulin2015Amultidimensional}.

In this paper, we consider the tropical optimization problems as to
\begin{equation*}
\begin{aligned}
&
\text{minimize}
&&
(\bm{A}\bm{x})^{-}\bm{x},
\end{aligned}
\qquad
\begin{aligned}
&
\text{minimize}
&&
\bm{x}^{-}\bm{A}\bm{x}
\oplus
(\bm{A}\bm{x})^{-}\bm{x},
\end{aligned}
\end{equation*}
where $\bm{A}$ is a given square matrix, $\bm{x}$ is an unknown vector, and the minus sign in the superscript serves to specify conjugate transposition of vectors. Since the objective function of the first problem is involved as a component of the composite objective function of the second, these problems are referred below to as component and composite problems for short.  

As applications, variants of these problems occur, for instance, in optimal project scheduling under the minimum flow time criterion \cite{Krivulin2015Extremal,Krivulin2015Amultidimensional,Krivulin2017Direct}, and in multicriteria decision making with pairwise comparisons \cite{Krivulin2016Using}.

Partial solutions of the problems were obtained in \cite{Krivulin2006Eigenvalues}, which specify a substantial part, but not all, of the solution sets for both irreducible and reducible matrices. The main purpose of this paper is to continue the investigation to derive complete solutions describing the entire solution set for general problems with arbitrary matrices. We follow the approach developed in \cite{Krivulin2015Soving} and based on a characterization of the solution set. We show that all solutions of the problems satisfy a vector inequality or a system of inequalities, and subsequently use these inequalities to establish characteristic properties of the solution set. The solutions are represented as a family of solution subsets, each defined by a matrix that is obtained by using a matrix sparsification technique.

Furthermore, we describe a backtracking procedure that allows one to reduce the brute-force generation of the matrices by skipping those, which cannot provide solutions, and thus offers an economical way to obtain all subsets in the family. Finally, the characteristic properties of the solution set are applied to provide a complete solution in a closed form. The results obtained are illustrated with illuminating numerical examples.

This paper further extends and supplements the results presented in the conference paper \cite{Krivulin2017Complete}, which examined only the problem with component objective function. The current paper further improves the presentation of the solution of the component problem, and offers new results on the solution of the composite problem.

The rest of the paper is organized as follows. In Section~\ref{S-PDR}, we give a brief overview of basic definitions and preliminary results of tropical algebra. Section~\ref{S-TOP} formulates the tropical optimization problems under study, and presents known solutions. In Section~\ref{S-OPCOF}, we investigate the first problem with a component objective function. As a result, a complete solution of the problem with reducible matrix is obtained in a compact vector form. The results obtained are then extended to the solution of the problem with composite objective function in Section~\ref{OPComOF}. Finally, Section~\ref{S-C} offers concluding remarks and suggestions for further research.

\section{Preliminary definitions and results}
\label{S-PDR}

We start with a brief overview of the preliminary definitions and results of tropical algebra to provide an appropriate formal background for the development of solutions for the tropical optimization problems in the subsequent sections. The overview is mainly based on the results in \cite{Krivulin2006Solution,Krivulin2006Eigenvalues,Krivulin2015Extremal,Krivulin2015Amultidimensional,Krivulin2017Direct}, which offer a useful framework to obtain solutions in a compact vector form, ready for further analysis and practical implementation. Additional details on tropical mathematics at both introductory and advanced levels can be found in many recent publications, including \cite{Kolokoltsov1997Idempotent,Golan2003Semirings,Heidergott2006Maxplus,Itenberg2007Tropical,Gondran2008Graphs,Butkovic2010Maxlinear,Mceneaney2010Maxplus,Maclagan2015Introduction}.  

\subsection{Idempotent semifield}

An \emph{idempotent semifield} is a system $(\mathbb{X},\mathbb{0},\mathbb{1},\oplus,\otimes)$, where $\mathbb{X}$ is a nonempty set endowed with associative and commutative operations, addition $\oplus$ and multiplication $\otimes$, which have as neutral elements the zero $\mathbb{0}$ and the one $\mathbb{1}$. Addition is idempotent, that is $x\oplus x=x$ for all $x\in\mathbb{X}$. Multiplication distributes over addition, has $\mathbb{0}$ as absorbing element, and is invertible, which gives any nonzero $x$ its inverse $x^{-1}$ such that $x\otimes x^{-1}=\mathbb{1}$.

Idempotent addition induces on $\mathbb{X}$ a partial order such that $x\leq y$ if and only if $x\oplus y=y$. With respect to this order, both addition and multiplication are monotone, which means that, for all $x,y,z\in\mathbb{X}$, the inequality $x\leq y$ entails that $x\oplus z\leq y\oplus z$ and $x\otimes z\leq y\otimes z$. Furthermore, inversion is antitone to take the inequality $x\leq y$ into $x^{-1}\geq y^{-1}$ for all nonzero $x$ and $y$. The inequality $x\oplus y\leq z$ is equivalent to the pair of inequalities $x\leq z$ and $y\leq z$. Finally, since $x\oplus\mathbb{0}=x$ implies that $x\geq\mathbb{0}$ for any $x$, the zero $\mathbb{0}$ is the least element of $\mathbb{X}$. The partial order is assumed to extend to a total order on the semifield.  

The power notation with integer exponents is routinely defined to represent iterated products for all $x\ne\mathbb{0}$ and integer $p\geq1$ in the form $x^{0}=\mathbb{1}$, $x^{p}=x\otimes x^{p-1}$, $x^{-p}=(x^{-1})^{p}$, and $\mathbb{0}^{p}=\mathbb{0}$. Moreover, the equation $x^{p}=a$ is assumed to have a unique solution $x=a^{1/p}$ for any $a$, which extends the notation to rational exponents. In what follows, the multiplication sign $\otimes$ is, as usual, dropped to save writing.

A typical example of a semifield is the system $(\mathbb{R}\cup\{-\infty\},-\infty,0,\max,+)$, which is usually referred to as the \emph{max-plus algebra}. In this semifield, the addition $\oplus$ is defined as $\max$, and the multiplication $\otimes$ is as arithmetic addition $+$. The number $-\infty$ is taken as the zero $\mathbb{0}$, and $0$ is as the one $\mathbb{1}$. For each $x\in\mathbb{R}$, the inverse $x^{-1}$ coincides with the conventional opposite number $-x$. For any rational $y$, the power $x^{y}$ corresponds to the arithmetic product $x\times y$. The order induced by idempotent addition complies with the natural linear order on $\mathbb{R}$.

\subsection{Matrix and vector algebra}

The set of matrices over $\mathbb{X}$ with $m$ rows and $n$ columns is denoted by $\mathbb{X}^{m\times n}$. A matrix with all entries equal to $\mathbb{0}$ is the \emph{zero matrix} denoted by $\bm{0}$. A matrix without zero rows (columns) is called row- (column-) regular.

For any matrices $\bm{A},\bm{B}\in\mathbb{X}^{m\times n}$ and $\bm{C}\in\mathbb{X}^{n\times l}$, and scalar $x\in\mathbb{X}$, matrix addition, matrix multiplication and scalar multiplication are routinely defined by the entry-wise formulas
$$
\{\bm{A}\oplus\bm{B}\}_{ij}
=
\{\bm{A}\}_{ij}
\oplus
\{\bm{B}\}_{ij},
\quad
\{\bm{A}\bm{C}\}_{ij}
=
\bigoplus_{k=1}^{n}\{\bm{A}\}_{ik}\{\bm{C}\}_{kj},
\quad
\{x\bm{A}\}_{ij}
=
x\{\bm{A}\}_{ij}.
$$

For any nonzero matrix $\bm{A}=(a_{ij})\in\mathbb{X}^{m\times n}$, the \emph{conjugate transpose} is the matrix $\bm{A}^{-}=(a_{ij}^{-})\in\mathbb{X}^{n\times m}$, where $a_{ij}^{-}=a_{ji}^{-1}$ if $a_{ji}\ne\mathbb{0}$, and $a_{ij}^{-}=\mathbb{0}$ otherwise.

The properties of scalar addition, multiplication and inversion with respect to the order relations are extended entry-wise to the matrix operations. 

Consider square matrices, which act as key components of the description and solution of the multidimensional optimization problems below. A square matrix is diagonal, if its off-diagonal entries are all equal to $\mathbb{0}$. A diagonal matrix with all diagonal entries equal to $\mathbb{1}$ is the \emph{identity matrix} denoted by $\bm{I}$. The power notation with non-negative integer exponents serves to represent repeated multiplication as $\bm{A}^{0}=\bm{I}$, $\bm{A}^{p}=\bm{A}\bm{A}^{p-1}$ and $\bm{0}^{p}=\bm{0}$ for any non-zero square matrix $\bm{A}$ and integer $p\geq1$.

If a row-regular matrix $\bm{A}$ has exactly one non-zero entry in each row, then the inequalities $\bm{A}^{-}\bm{A}\leq\bm{I}$ and $\bm{A}\bm{A}^{-}\geq\bm{I}$ hold (corresponding, in the context of relational algebra, to the univalent and total properties of a relation $\bm{A}$).

The \emph{trace} of a square matrix $\bm{A}=(a_{ij})\in\mathbb{X}^{n\times n}$ of order $n$ is routinely defined as
$$
\mathop\mathrm{tr}\bm{A}
=
a_{11}\oplus\cdots\oplus a_{nn}
=
\bigoplus_{i=1}^{n}a_{ii},
$$
and retains the standard properties of traces with respect to matrix addition and to matrix and scalar multiplications.   

To represent solutions proposed in the subsequent sections, we exploit the function, which takes any square matrix $\bm{A}\in\mathbb{X}^{n\times n}$ to the scalar
\begin{equation*}
\mathop\mathrm{Tr}(\bm{A})
=
\mathop\mathrm{tr}\bm{A}\oplus\cdots\oplus\mathop\mathrm{tr}\bm{A}^{n}
=
\bigoplus_{m=1}^{n}
\mathop\mathrm{tr}\bm{A}^{m}.
\end{equation*}

Provided that the condition $\mathop\mathrm{Tr}(\bm{A})\leq\mathbb{1}$ holds, the \emph{asterisk and plus operators} (also known as the Kleene star and Kleene plus) map $\bm{A}$ to the matrices
\begin{equation*}
\bm{A}^{\ast}
=
\bm{I}\oplus\bm{A}\oplus\cdots\oplus\bm{A}^{n-1}
=
\bigoplus_{m=0}^{n-1}\bm{A}^{m},
\qquad
\bm{A}^{+}
=
\bm{A}\oplus\cdots\oplus\bm{A}^{n}
=
\bigoplus_{m=1}^{n}\bm{A}^{m}.
\end{equation*}

If $\mathop\mathrm{Tr}(\bm{A})\leq\mathbb{1}$, then the inequality $\bm{A}^{k}\leq\bm{A}^{\ast}$ holds for all integers $k\geq0$. As a consequence, the inequality $\bm{A}^{+}=\bm{A}\bm{A}^{\ast}\leq\bm{A}^{\ast}$ is valid as well. 

The description of the solutions also involves the matrix $\bm{A}^{\times}$ which is obtained from $\bm{A}$ as follows. First, we assume that $\mathop\mathrm{Tr}(\bm{A})\leq\mathbb{1}$, and calculate the matrix $\bm{A}^{+}$. Then, $\bm{A}^{\times}$ is constructed by taking those columns in the matrix $\bm{A}^{+}$ which have their diagonal entries equal to $\mathbb{1}$, and thus, in general, can be a non-square matrix.

Any matrix that consists of one row (column) is considered a row (column) vector. All vectors are assumed to be column vectors, unless otherwise specified. The set of column vectors of order $n$ is denoted $\mathbb{X}^{n}$. A vector with all zero elements is the zero vector $\bm{0}$. A vector is regular if it has no zero elements.

For any non-zero vector $\bm{x}=(x_{j})\in\mathbb{X}^{n}$, the conjugate transpose is the row vector $\bm{x}^{-}=(x_{j}^{-})$, where $x_{j}^{-}=x_{j}^{-1}$ if $x_{j}\ne\mathbb{0}$, and $x_{j}^{-}=\mathbb{0}$ otherwise.

For any non-zero vector $\bm{x}$, the equality $\bm{x}^{-}\bm{x}=\mathbb{1}$ is obviously valid.

For any regular vectors $\bm{x},\bm{y}\in\mathbb{X}^{n}$, the matrix inequality $\bm{x}\bm{y}^{-}\geq(\bm{x}^{-}\bm{y})^{-1}\bm{I}$ holds and becomes $\bm{x}\bm{x}^{-}\geq\bm{I}$ when $\bm{y}=\bm{x}$.

A vector $\bm{b}$ is said to be linearly dependent on vectors $\bm{a}_{1},\ldots,\bm{a}_{n}$ if the equality $\bm{b}=x_{1}\bm{a}_{1}\oplus\cdots\oplus x_{n}\bm{a}_{n}$ holds for some scalars $x_{1},\ldots,x_{n}$. The vector $\bm{b}$ is linearly dependent on $\bm{a}_{1},\ldots,\bm{a}_{n}$ if and only if the condition $(\bm{A}(\bm{b}^{-}\bm{A})^{-})^{-}\bm{b}=\mathbb{1}$ is valid, where $\bm{A}$ is the matrix with the vectors $\bm{a}_{1},\ldots,\bm{a}_{n}$ as its columns.

A system of vectors $\bm{a}_{1},\ldots,\bm{a}_{n}$ is linearly dependent if at least one vector is linearly dependent on others, and linearly independent otherwise.

Suppose that the system $\bm{a}_{1},\ldots,\bm{a}_{n}$ is linearly dependent. To construct a maximal linearly independent system, we use a procedure that sequentially reduces the system until it becomes linearly independent. The procedure applies the above condition to examine the vectors one by one to remove a vector if it is linearly dependent on others, or to leave the vector in the system otherwise. 

A scalar $\lambda\in\mathbb{X}$ is an \emph{eigenvalue} and a non-zero vector $\bm{x}\in\mathbb{X}^{n}$ is a corresponding \emph{eigenvector} of a square matrix $\bm{A}\in\mathbb{X}^{n\times n}$ if they satisfy the equality
$$
\bm{A}\bm{x}
=
\lambda\bm{x}.
$$

\subsection{Reducible and irreducible matrices}

A matrix $\bm{A}\in\mathbb{X}^{n\times n}$ is reducible if simultaneous permutations of its rows and columns can transform it into a block-triangular normal form, and irreducible otherwise. The \emph{lower block-triangular normal form} of the matrix $\bm{A}$ is given by
\begin{equation}
\bm{A}
=
\left(
\begin{array}{cccc}
\bm{A}_{11} & \bm{0} & \ldots & \bm{0}
\\
\bm{A}_{21} & \bm{A}_{22} & & \bm{0}
\\
\vdots & \vdots & \ddots &
\\
\bm{A}_{s1} & \bm{A}_{s2} & \ldots & \bm{A}_{ss}
\end{array}
\right),
\label{E-MNF}
\end{equation}
where, in each block row $i=1,\ldots,s$, the diagonal block $\bm{A}_{ii}$ is either irreducible or the zero square matrix of order $n_{i}$, the off-diagonal blocks $\bm{A}_{ij}$ are arbitrary matrices of size $n_{i}\times n_{j}$ for all $j<i$, and $n_{1}+\cdots+n_{s}=n$.

Any irreducible matrix $\bm{A}$ has only one eigenvalue, which is calculated as
\begin{equation}
\lambda
=
\mathop\mathrm{tr}\bm{A}
\oplus\cdots\oplus
\mathop\mathrm{tr}\nolimits^{1/n}(\bm{A}^{n})
=
\bigoplus_{m=1}^{n}\mathop\mathrm{tr}\nolimits^{1/m}(\bm{A}^{m}).
\label{E-lambda-tr1mAm}
\end{equation}

From \eqref{E-lambda-tr1mAm} it follows, in particular, that $\mathrm{tr}(\bm{A}^{m})\leq\lambda^{m}$ for all $m=1,\ldots,n$.

All eigenvectors of the irreducible matrix $\bm{A}$ are regular, and given by
$$
\bm{x}
=
(\lambda^{-1}\bm{A})^{\times}\bm{u},
$$
where $\bm{u}$ is any regular vector of appropriate size.

Note that every irreducible matrix is both row- and column-regular.

Let $\bm{A}$ be a matrix represented in the form \eqref{E-MNF}. Denote by $\lambda_{i}$ the eigenvalue of the diagonal block $\bm{A}_{ii}$ for $i=1,\ldots,s$. Then, the scalar $\lambda=\lambda_{1}\oplus\cdots\oplus\lambda_{s}$ is the maximum eigenvalue of the matrix $\bm{A}$, which is referred to as the spectral radius of $\bm{A}$ and calculated as \eqref{E-lambda-tr1mAm}. For any irreducible matrix, the spectral radius coincides with the unique eigenvalue of the matrix.

Without loss of generality, the normal form \eqref{E-MNF} can be assumed to order all block rows, which have non-zero blocks on the diagonal and zero blocks elsewhere, before the block rows with non-zero off-diagonal blocks. Moreover, the rows, which have non-zero blocks only on the diagonal, can be arranged in increasing order of the eigenvalues of diagonal blocks. Then, the normal form is refined as
\begin{equation}
\bm{A}
=
\left(
\begin{array}{cccccc}
\bm{A}_{11} & & \bm{0} & \bm{0} & \ldots & \bm{0}
\\
& \ddots & & \vdots & & \vdots
\\
\bm{0} & & \bm{A}_{rr} & \bm{0} & \ldots & \bm{0}
\\
\bm{A}_{r+1,1} & \ldots & \bm{A}_{r+1,r} & \bm{A}_{r+1,r+1} & & \bm{0}
\\
\vdots & & \vdots & \vdots & \ddots
\\
\bm{A}_{s1} & \ldots & \bm{A}_{sr} & \bm{A}_{s,r+1} & \ldots & \bm{A}_{ss}
\end{array}
\right),
\label{E-MNFr}
\end{equation}
where the eigenvalues of $\bm{A}_{11},\ldots,\bm{A}_{rr}$ satisfy the condition $\lambda_{1}\leq\cdots\leq\lambda_{r}$, and each row $i=r+1,\ldots,s$ has a block $\bm{A}_{ij}\ne\bm{0}$ for some $j<i$.

\subsection{Vector inequalities and equations}

In this subsection, we present solutions to vector inequalities, which appear below in the analysis of the optimization problems under study. 

Suppose that, given a matrix $\bm{A}\in\mathbb{X}^{m\times n}$ and vector $\bm{d}\in\mathbb{X}^{m}$, we need to find vectors $\bm{x}\in\mathbb{X}^{n}$ to satisfy the inequality
\begin{equation}
\bm{A}\bm{x}
\leq
\bm{d}.
\label{I-Axleqd}
\end{equation}

A direct solution proposed in \cite{Krivulin2015Extremal} can be obtained as follows.
\begin{lemma}
\label{L-Axleqd}
For any column-regular matrix $\bm{A}$ and regular vector $\bm{d}$, all solutions to inequality \eqref{I-Axleqd} are given by the inequality $\bm{x}\leq(\bm{d}^{-}\bm{A})^{-}$.
\end{lemma}

Next, we consider the following problem: given a matrix $\bm{A}\in\mathbb{X}^{n\times n}$, find regular vectors $\bm{x}\in\mathbb{X}^{n}$ to satisfy the inequality
\begin{equation}
\bm{A}\bm{x}
\leq
\bm{x}.
\label{I-Axleqx}
\end{equation}

%Note that the inequality always has the trivial solution $\bm{x}=\bm{0}$.

The following result \cite{Krivulin2006Solution,Krivulin2015Amultidimensional} offers a direct solution to inequality \eqref{I-Axleqx}.
\begin{theorem}
\label{T-Axleqx}
For any matrix $\bm{A}$, the following statements hold:
\begin{enumerate}
\item If $\mathop\mathrm{Tr}(\bm{A})\leq\mathbb{1}$, then all regular solutions to \eqref{I-Axleqx} are given by $\bm{x}=\bm{A}^{\ast}\bm{u}$, where $\bm{u}$ is any regular vector.
\item If $\mathop\mathrm{Tr}(\bm{A})>\mathbb{1}$, then there is no regular solution.
\end{enumerate}
\end{theorem}

We conclude this subsection with a solution to a vector equation. Given a matrix $\bm{A}\in\mathbb{X}^{n\times n}$ and a vector $\bm{b}\in\mathbb{X}^{n}$, the problem is to find regular vectors $\bm{x}\in\mathbb{X}^{n}$ that solve the equation
\begin{equation}
\bm{A}\bm{x}\oplus\bm{b}
=
\bm{x}.
\label{E-Axbeqx}
\end{equation}

The next statement \cite{Krivulin2006Solution} offers a solution when the matrix $\bm{A}$ is irreducible.
\begin{theorem}
\label{T-Axbeqx}
For any irreducible matrix $\bm{A}$ and non-zero vector $\bm{b}$, the following statements hold:
\begin{enumerate}
\item If $\mathop\mathrm{Tr}(\bm{A})<\mathbb{1}$, then equation \eqref{E-Axbeqx} has the unique regular solution $\bm{x}=\bm{A}^{\ast}\bm{b}$.
\item If $\mathop\mathrm{Tr}(\bm{A})=\mathbb{1}$, then all regular solutions to \eqref{E-Axbeqx} are given by $\bm{x}=\bm{A}^{\ast}\bm{b}\oplus\bm{A}^{\times}\bm{u}$, where $\bm{u}$ is any regular vector of appropriate size.
\item If $\mathop\mathrm{Tr}(\bm{A})>\mathbb{1}$, then there is no regular solution.
\end{enumerate}
\end{theorem}

\section{Tropical optimization problems}
\label{S-TOP}

We are now in a position to describe the optimization problems under study, and to provide some preliminary solution to the problems. Given a matrix $\bm{A}\in\mathbb{X}^{n\times n}$, the problems are formulated to minimize different objective functions that are defined on vectors $\bm{x}\in\mathbb{X}^{n}$ through the matrix $\bm{A}$ by conjugate transposition.

We start with the problem to find regular vectors $\bm{x}$ that 
\begin{equation}
\begin{aligned}
&
\text{minimize}
&&
\bm{x}^{-}\bm{A}\bm{x}.
%\oplus
%(\bm{A}\bm{x})^{-}\bm{x}.
\end{aligned}
\label{P-minxAx}
\end{equation}

A complete direct solution of the problem is obtained in \cite{Krivulin2015Extremal,Krivulin2015Amultidimensional} as follows.
\begin{lemma}\label{L-minxAx}
Let $\bm{A}$ be a matrix with spectral radius $\lambda>\mathbb{0}$. Then, the minimum value in problem \eqref{P-minxAx} is equal to $\lambda$, and all regular solutions are given by $\bm{x}=(\lambda^{-1}\bm{A})^{\ast}\bm{u}$, where $\bm{u}$ is any regular vector.
\end{lemma}

Assume now that we need to find regular solutions to the problem
\begin{equation}
\begin{aligned}
&
\text{minimize}
&&
%\bm{x}^{-}\bm{A}\bm{x}
%\oplus
(\bm{A}\bm{x})^{-}\bm{x}.
\end{aligned}
\label{P-minAxx}
\end{equation}

The next result describes a partial solution given in \cite{Krivulin2006Eigenvalues} to the problem when the matrix $\bm{A}$ is irreducible.
\begin{lemma}\label{L-minAxx}
Let $\bm{A}$ be an irreducible matrix with spectral radius (eigenvalue) $\lambda$. Then, the minimum value in problem \eqref{P-minAxx} is equal to $\lambda^{-1}$, and attained at the eigenvectors of the matrix $\bm{A}$, given by $\bm{x}=(\lambda^{-1}\bm{A})^{\times}\bm{u}$, where $\bm{u}$ is any regular vector.
\end{lemma}

Finally, we consider the problem with a composite objective function that uses the objective functions in \eqref{P-minxAx} and \eqref{P-minAxx} as components. The problem is formulated to find regular vectors $\bm{x}$ that 
\begin{equation}
\begin{aligned}
&
\text{minimize}
&&
\bm{x}^{-}\bm{A}\bm{x}
\oplus
(\bm{A}\bm{x})^{-}\bm{x}.
\end{aligned}
\label{P-minxAxAxx}
\end{equation}

A solution to the problem for irreducible matrices $\bm{A}$ is provided by the following statement \cite{Krivulin2006Eigenvalues}.
\begin{lemma}\label{L-minxAxAxx}
Let $\bm{A}$ be an irreducible matrix with spectral radius (eigenvalue) $\lambda$. Then, the minimum value in problem \eqref{P-minxAxAxx} is equal to $\lambda\oplus\lambda^{-1}$, and attained at the eigenvectors of the matrix $\bm{A}$, given by $\bm{x}=(\lambda^{-1}\bm{A})^{\times}\bm{u}$, where $\bm{u}$ is any regular vector.
\end{lemma}

In the subsequent sections, we extend results obtained for problems \eqref{P-minAxx} and \eqref{P-minxAxAxx} with component and composite objective functions to reducible matrices, and offer complete solutions to the problems. We start with a statement that determines the minimum in problem \eqref{P-minAxx} with an arbitrary matrix $\bm{A}$, reduces the problem to a vector inequality, and offers an explicit representation of a subset of solutions. Furthermore, to derive a complete solution to problem \eqref{P-minAxx}, we use a matrix sparsification technique to find all solutions of the inequality in the form of a family of subsets, each of which is given by a sparsified matrix obtained from $\bm{A}$. We suggest a procedure that reduces the brute-force generation of sparsified matrices by skipping those, which do not lead to solutions. Finally, we show how to represent the complete solution of \eqref{P-minAxx} in a compact vector form.

The techniques developed and results obtained for problem \eqref{P-minAxx} are then extended to solve problem \eqref{P-minxAxAxx} in a similar way.

\section{Optimization problem with component objective function}
\label{S-OPCOF}

We now turn to our new results to extend the solution of Lemma~\ref{L-minAxx} to arbitrary matrices, and then to obtain two useful consequences. For simplicity, we concentrate on the matrices in refined block-triangular form \eqref{E-MNFr}, which have no zero rows. The case of matrices with zero rows follows the same arguments with minor technical modifications.

\begin{theorem}
\label{T-minAxx}
Let $\bm{A}$ be a matrix in the refined block-triangular normal form \eqref{E-MNFr}, where the diagonal block $\bm{A}_{11}$ has eigenvalue $\lambda_{1}>\mathbb{0}$. 

Then, the minimum value in problem \eqref{P-minAxx} is equal to $\lambda_{1}^{-1}$, and all regular solutions are characterized by the inequality
\begin{equation}
\bm{x}
\leq
\lambda_{1}^{-1}\bm{A}\bm{x}.
\label{I-xleqlambda1Ax}
\end{equation}

Specifically, any block vector $\bm{x}^{T}=(\bm{x}_{1}^{T},\ldots,\bm{x}_{s}^{T})$ with the blocks $\bm{x}_{i}$ of order $n_{i}$, defined successively for each $i=1,\ldots,s$ by the conditions 
\begin{equation*}
\bm{x}_{i}
=
\begin{cases}
(\lambda_{i}^{-1}\bm{A}_{ii})^{\times}\bm{u}_{i},
&
\text{if $\lambda_{i}\geq\lambda_{1}$};
\\
\lambda_{1}^{-1}(\lambda_{1}^{-1}\bm{A}_{ii})^{\ast}\displaystyle\bigoplus_{j=1}^{i-1}\bm{A}_{ij}\bm{x}_{j},
&
\text{if $\lambda_{i}<\lambda_{1}$};
\end{cases}
\end{equation*}
where $\bm{u}_{i}$ are regular vectors of appropriate size, is a solution of the problem.
\end{theorem}
\begin{proof}
Let $\bm{x}^{T}=(\bm{x}_{1}^{T},\ldots,\bm{x}_{s}^{T})$ be an arbitrary regular vector in block form, where $\bm{x}_{i}$ is a vector of order $n_{i}$ for each $i=1,\ldots,n$. Considering the refined block-triangular normal form of the matrix $\bm{A}$ with the condition that $\lambda_{i}\geq\lambda_{1}$ for $i\leq r$, we apply Lemma~\ref{L-minAxx} to write
\begin{equation*}
(\bm{A}\bm{x})^{-}\bm{x}
=
\bigoplus_{i=1}^{s}\left(\bigoplus_{j=1}^{i}\bm{A}_{ij}\bm{x}_{j}\right)^{-}\bm{x}_{i}
\geq
\bigoplus_{i=1}^{r}(\bm{A}_{ii}\bm{x}_{i})^{-}\bm{x}_{i}
\geq
\lambda_{1}^{-1}\oplus\cdots\oplus\lambda_{r}^{-1}
=
\lambda_{1}^{-1},
\end{equation*}
which means that $\lambda_{1}^{-1}$ is a lower bound for the objective function.

Furthermore, we show that the bound $\lambda_{1}^{-1}$ is attained at some regular vector $\bm{x}$, and hence it is the minimum in the problem. To find such a vector, we solve the inequality $(\bm{A}\bm{x})^{-}\bm{x}\leq\lambda_{1}^{-1}$, which is equivalent to the system of inequalities
\begin{equation*}
\left(\bigoplus_{j=1}^{i}\bm{A}_{ij}\bm{x}_{j}\right)^{-}\bm{x}_{i}
\leq
\lambda_{1}^{-1},
\qquad
i=1,\ldots,s.
\end{equation*}

We successively define a sequence of vectors $\bm{x}_{i}$ for $i=1,\ldots,s$. If $\lambda_{i}\geq\lambda_{1}$ we take $\bm{x}_{i}$ to be an eigenvector of $\bm{A}_{ii}$, given by the equation $\bm{A}_{ii}\bm{x}_{i}=\lambda_{i}\bm{x}_{i}$, which is solved as
\begin{equation*}
\bm{x}_{i}
=
(\lambda_{i}^{-1}\bm{A}_{ii})^{\times}\bm{u}_{i},
\end{equation*}
where $\bm{u}_{i}$ is a regular vector of appropriate size.

Note that the condition $\lambda_{i}\geq\lambda_{1}$ is fulfilled if $i\leq r$. With this condition, we have
\begin{equation*}
\left(\bigoplus_{j=1}^{i}\bm{A}_{ij}\bm{x}_{j}\right)^{-}\bm{x}_{i}
\leq
(\bm{A}_{ii}\bm{x}_{i})^{-}\bm{x}_{i}
=
\lambda_{i}^{-1}
\leq
\lambda_{1}^{-1}.
\end{equation*}

Provided that $\lambda_{i}<\lambda_{1}$, we define the vector $\bm{x}_{i}$ as a solution of the equation $\lambda_{1}^{-1}(\bm{A}_{i1}\bm{x}_{1}\oplus\cdots\oplus\bm{A}_{ii}\bm{x}_{i})=\bm{x}_{i}$. 

Since, in this case, $\mathop\mathrm{Tr}(\lambda_{1}^{-1}\bm{A}_{ii})=\lambda_{1}^{-1}\mathop\mathrm{tr}\bm{A}_{ii}\oplus\cdots\oplus\lambda_{1}^{-n}\mathop\mathrm{tr}(\bm{A}_{ii}^{n})<\mathbb{1}$, the equation is solved by Theorem~\ref{T-Axbeqx} in the form
\begin{equation*}
\bm{x}_{i}
=
\lambda_{1}^{-1}(\lambda_{1}^{-1}\bm{A}_{ii})^{\ast}\bigoplus_{j=1}^{i-1}\bm{A}_{ij}\bm{x}_{j}.
\end{equation*} 

With the solution vector $\bm{x}_{i}$, we have
\begin{equation*}
\left(\bigoplus_{j=1}^{i}\bm{A}_{ij}\bm{x}_{j}\right)^{-}\bm{x}_{i}
=
\lambda_{1}^{-1}.
\end{equation*}

Finally, the application of Lemma~\ref{L-Axleqd} to solve the inequality $(\bm{A}\bm{x})^{-}\bm{x}\leq\lambda_{1}^{-1}$ with respect to $\bm{x}$ leads to inequality \eqref{I-xleqlambda1Ax}.
\end{proof}

We now illustrate the above result with the solution of an example problem with a reducible matrix.

\begin{example}
\label{X-rmatrix}
Consider problem \eqref{P-minAxx}, which is formulated in terms of the idempotent semifield $\mathbb{R}_{\max,+}$ with the objective function given by the matrix
\begin{equation}
\bm{A}
=
\left(
\begin{array}{ccr}
1 & \mathbb{0} & \mathbb{0}
\\
3 & 2 & \mathbb{0}
\\
\mathbb{0} & 0 & -1
\end{array}
\right),
\label{E-A100320001}
\end{equation}
where the notation $\mathbb{0}=-\infty$ is used to save writing.

Note that the matrix $\bm{A}$ is reducible, and has the block-triangular form \eqref{E-MNFr} with the blocks defined by $(1\times1)$-matrices as
\begin{equation*}
\bm{A}_{11}
=
(1),
\quad
\bm{A}_{21}
=
(3),
\quad
\bm{A}_{22}
=
(2),
\quad
\bm{A}_{31}
=
(\mathbb{0}),
\quad
\bm{A}_{32}
=
(0),
\quad
\bm{A}_{33}
=
(-1).
\end{equation*}

The eigenvalues of diagonal blocks $\bm{A}_{11}$, $\bm{A}_{22}$ and $\bm{A}_{33}$ are easily found to be $\lambda_{1}=1$, $\lambda_{2}=2$ and $\lambda_{3}=-1$, whereas their corresponding eigenvectors take the form of arbitrary reals.

By Theorem~\ref{T-minAxx}, the minimum in the problem is equal to $\lambda_{1}^{-1}=-1$. The solution offered by the theorem is given by the vector $\bm{x}=(x_{1},x_{2},x_{3})^{T}$, where $x_{1}=u_{1}$ and $x_{2}=u_{2}$ for any $u_{1},u_{2}\in\mathbb{R}$.

Furthermore, we have $x_{3}=\lambda_{1}^{-1}(\lambda_{1}^{-1}\bm{A}_{33})^{\ast}\bm{A}_{32}x_{2}=(-1)x_{2}=(-1)u_{2}$, or, in the usual notation, $x_{3}=u_{2}-1$. 

In vector form, with the notation $\bm{u}=(u_{1},u_{2})^{T}$, the solution becomes
\begin{equation*}
\bm{x}
=
\left(
\begin{array}{cr}
0 & \mathbb{0}
\\
\mathbb{0} & 0
\\
\mathbb{0} & -1
\end{array}
\right)
\bm{u},
\qquad
\bm{u}\in\mathbb{R}^{2}.
\qed
\end{equation*}
\end{example}

We now consider a special case of the problem, where the partial solution given by the previous theorem takes a more compact form.

\begin{corollary}
Under the conditions of Theorem~\ref{T-minAxx}, if $\lambda_{1}\leq\lambda_{i}$ for all $i=1,\ldots,s$, then the vector
$$
\bm{x}
=
\bm{D}\bm{u},
\qquad
\bm{D}
=
\left(
\begin{array}{ccc}
(\lambda_{1}^{-1}\bm{A}_{11})^{\times} & & \bm{0}
\\
& \ddots &
\\
\bm{0} & & (\lambda_{s}^{-1}\bm{A}_{ss})^{\times}
\end{array}
\right),
$$
where $\bm{u}$ is any regular vector of appropriate size, is a solution of the problem.
\end{corollary}
\begin{proof}
It follows from Theorem~\ref{T-minAxx} that the vector $\bm{x}^{T}=(\bm{x}_{1}^{T},\ldots,\bm{x}_{s}^{T})$, which has, for all $i=1,\ldots,s$, the blocks $\bm{x}_{i}=(\lambda_{i}^{-1}\bm{A}_{ii})^{\times}\bm{u}_{i}$, where $\bm{u}_{i}$ are regular vectors of appropriate size, is a solution of the problem.

It remains to introduce the block vector $\bm{u}^{T}=(\bm{u}_{1}^{T},\ldots,\bm{u}_{s}^{T})$ and the block-diagonal matrix $\bm{D}=\mathop\mathrm{diag}((\lambda_{1}^{-1}\bm{A}_{11})^{\times},\ldots,(\lambda_{s}^{-1}\bm{A}_{ss})^{\times})$ to finish the proof.
\end{proof}

The next result shows a useful property of the solutions of problem \eqref{P-minAxx}.
\begin{corollary}
\label{C-CuVASM}
Under the conditions of Theorem~\ref{T-minAxx}, the set of solution vectors of problem \eqref{P-minAxx} is closed under vector addition and scalar multiplication.
\end{corollary}
\begin{proof}
Suppose that vectors $\bm{x}$ and $\bm{y}$ are solutions of the problem, which implies, by Theorem~\ref{T-minAxx}, that $\bm{x}\leq\lambda_{1}^{-1}\bm{A}\bm{x}$ and $\bm{y}\leq\lambda_{1}^{-1}\bm{A}\bm{y}$. We take arbitrary scalars $\alpha$ and $\beta$, and consider the vector $\bm{z}=\alpha\bm{x}\oplus\beta\bm{y}$. Since 
$$
\bm{z}
=
\alpha\bm{x}\oplus\beta\bm{y}
\leq
\alpha\lambda_{1}^{-1}\bm{A}\bm{x}\oplus\beta\lambda_{1}^{-1}\bm{A}\bm{y}
=
\lambda_{1}^{-1}\bm{A}(\alpha\bm{x}\oplus\beta\bm{y})
=
\lambda_{1}^{-1}\bm{A}\bm{z},
$$
the vector $\bm{z}$ is a solution of the problem, which proves the statement.
\end{proof}

\subsection{Derivation of complete solution}

It follows from the results from the previous section that, under the assumptions of Theorem~\ref{T-minAxx}, all solutions of problem \eqref{P-minAxx} are given by inequality \eqref{I-xleqlambda1Ax}. Below, we derive all solutions of the inequality in the form of a family of solution sets, each defined by means of sparsification of the matrix $\bm{A}$.

\begin{theorem}
\label{T-xlambda1Ax}
Let $\bm{A}$ be a matrix in the refined block-triangular normal form \eqref{E-MNFr}, where the diagonal block $\bm{A}_{11}$ has eigenvalue $\lambda_{1}>\mathbb{0}$.

Denote by $\mathcal{A}$ the set of matrices $\bm{A}_{k}$ that are obtained from $\bm{A}$ by fixing one non-zero entry in each row and by setting the others to $\mathbb{0}$, and that satisfy the condition $\mathop\mathrm{Tr}(\bm{B}_{k})\leq\mathbb{1}$, where $\bm{B}_{k}=\bm{A}_{k}^{-}(\bm{A}\oplus\lambda_{1}\bm{I})$.

Then, all regular solutions of inequality \eqref{I-xleqlambda1Ax} are given by the conditions
\begin{equation}
\bm{x}
=
\bm{B}_{k}^{\ast}\bm{u},
\qquad
\bm{B}_{k}
=
\bm{A}_{k}^{-}(\bm{A}\oplus\lambda_{1}\bm{I}),
\qquad
\bm{A}_{k}
\in
\mathcal{A},
\qquad
\bm{u}
>
\bm{0}.
\label{C-xAkAIu-u0-AkA}
\end{equation}
\end{theorem}
\begin{proof}
First we note that, under the conditions of the theorem, regular solutions to inequality \eqref{I-xleqlambda1Ax} exist. Indeed, using similar arguments as in Theorem~\ref{T-minAxx}, one can see that the block vector $\bm{x}^{T}=(\bm{x}_{1}^{T},\ldots,\bm{x}_{s}^{T})$, where, for all $i=1,\ldots,s$, the block $\bm{x}_{i}$ is an eigenvector of the matrix $\bm{A}_{ii}$ if $\lambda_{i}\geq\mathbb{1}$, or a solution of the equation $\bm{A}_{i1}\bm{x}_{1}\oplus\cdots\oplus\bm{A}_{ii}\bm{x}_{i}=\lambda_{1}\bm{x}_{i}$ otherwise, satisfies the inequality.

To prove the theorem, we show that any regular solution of inequality \eqref{I-xleqlambda1Ax} can be represented as \eqref{C-xAkAIu-u0-AkA}, and vice versa. Assume $\bm{x}=(x_{j})$ to be a regular solution of \eqref{I-xleqlambda1Ax} with a matrix $\bm{A}=(a_{ij})$, and consider the scalar inequality
\begin{equation}
x_{p}
\leq
\lambda_{1}^{-1}a_{p1}x_{1}\oplus\cdots\oplus\lambda_{1}^{-1}a_{pn}x_{n},
\label{I-xp}
\end{equation}
which corresponds to row $p$ in the matrix $\bm{A}$.

If this inequality holds for some $x_{1},\ldots,x_{n}$, then, as the order defined by the relation $\leq$ is assumed linear, there is a term in the sum on the right-hand side that provides the maximum of the sum. Suppose that the maximum is attained at the $q$th term $\lambda_{1}^{-1}a_{pq}x_{q}$, and hence $a_{pq}>\mathbb{0}$. Under this condition, we can replace the above inequality by the two inequalities $\lambda_{1}^{-1}a_{pq}x_{q}\geq\lambda_{1}^{-1}a_{p1}x_{1}\oplus\cdots\oplus\lambda_{1}^{-1}a_{pn}x_{n}$ and $\lambda_{1}^{-1}a_{pq}x_{q}\geq x_{p}$, or, equivalently, by one inequality
\begin{equation}
\lambda_{1}^{-1}a_{pq}x_{q}
\geq
\lambda_{1}^{-1}a_{p1}x_{1}\oplus\cdots\oplus(\lambda_{1}^{-1}a_{pp}\oplus\mathbb{1})x_{p}\oplus\cdots\oplus\lambda_{1}^{-1}a_{pn}x_{n}.
\label{I-apqxq}
\end{equation}

Now assume that we determine maximum terms in all scalar inequalities in \eqref{I-xleqlambda1Ax}. Similarly as above, we replace each inequality by an inequality with the maximum term isolated on the left-hand side.

To represent these scalar inequalities in a vector form, we introduce a matrix $\bm{A}_{k}$ that is obtained from $\bm{A}$ by fixing, in each row, one entry, which corresponds to the maximum term, and by setting the other entries to $\mathbb{0}$. With this matrix $\bm{A}_{k}$, the scalar inequalities combine into the vector inequality $\lambda_{1}^{-1}\bm{A}_{k}\bm{x}\geq(\lambda_{1}^{-1}\bm{A}\oplus\bm{I})\bm{x}$.

Let us verify that the last inequality is equivalent to the inequality
\begin{equation*}
\bm{x}
\geq
\lambda_{1}\bm{A}_{k}^{-}(\lambda_{1}^{-1}\bm{A}\oplus\bm{I})\bm{x}.
\end{equation*} 

We multiply the former inequality by $\lambda_{1}\bm{A}_{k}^{-}$ on the left, and take into account that $\bm{A}_{k}^{-}\bm{A}_{k}\leq\bm{I}$ to write $\bm{x}\geq\bm{A}_{k}^{-}\bm{A}_{k}\bm{x}\geq\lambda_{1}\bm{A}_{k}^{-}(\lambda_{1}^{-1}\bm{A}\oplus\bm{I})\bm{x}$, which gives the latter one. At the same time, the multiplication of the latter inequality by $\lambda^{-1}\bm{A}_{k}$ on the left, and the condition $\bm{A}_{k}\bm{A}_{k}^{-}\geq\bm{I}$ result in the former inequality as $\lambda_{1}^{-1}\bm{A}_{k}\bm{x}\geq\bm{A}_{k}\bm{A}_{k}^{-}(\lambda_{1}^{-1}\bm{A}\oplus\bm{I})\bm{x}\geq(\lambda_{1}^{-1}\bm{A}\oplus\bm{I})\bm{x}$.

With the matrix $\bm{B}_{k}=\lambda_{1}\bm{A}_{k}^{-}(\lambda_{1}^{-1}\bm{A}\oplus\bm{I})=\bm{A}_{k}^{-}(\bm{A}\oplus\lambda_{1}\bm{I})$, the inequality obtained becomes
\begin{equation*}
\bm{x}
\geq
\bm{B}_{k}\bm{x}.
\end{equation*} 

By assumption, the last inequality has a regular solution $\bm{x}$, which, according to Theorem~\ref{T-Axleqx}, implies that the condition $\mathop\mathrm{Tr}(\bm{B}_{k})\leq\mathbb{1}$ holds. Under this condition, any regular solution is given by $\bm{x}=\bm{B}_{k}^{\ast}\bm{u}$, where $\bm{u}$ is a regular vector, which means that the vector $\bm{x}$ is represented in the form of \eqref{C-xAkAIu-u0-AkA}.

Now suppose that a vector $\bm{x}$ is defined by the conditions at \eqref{C-xAkAIu-u0-AkA}. To verify that $\bm{x}$ satisfies \eqref{I-xleqlambda1Ax}, we first use the condition $\mathop\mathrm{Tr}(\bm{B}_{k})\leq\mathbb{1}$ which entails $\bm{B}_{k}^{+}=\bm{B}_{k}\bm{B}_{k}^{\ast}\leq\bm{B}_{k}^{\ast}$. Next, we note that $\bm{B}_{k}=\bm{A}_{k}^{-}(\bm{A}\oplus\lambda_{1}\bm{I})\geq\lambda_{1}\bm{A}_{k}^{-}$. Considering that from $\bm{A}\geq\bm{A}_{k}$ it follows that $\bm{A}\bm{A}_{k}^{-}\geq\bm{A}_{k}\bm{A}_{k}^{-}\geq\bm{I}$, we write
\begin{equation*}
\bm{A}\bm{B}_{k}^{\ast}
\geq
\bm{A}\bm{B}_{k}^{+}
=
\bm{A}\bm{B}_{k}\bm{B}_{k}^{\ast}
\geq
\lambda_{1}\bm{A}\bm{A}_{k}^{-}\bm{B}_{k}^{\ast}
\geq
\lambda_{1}\bm{B}_{k}^{\ast}.
\end{equation*}

In this case, we have
\begin{equation*}
\lambda_{1}^{-1}\bm{A}\bm{x}
=
\lambda_{1}^{-1}\bm{A}\bm{B}_{k}^{\ast}\bm{u}
\geq
\lambda_{1}^{-1}\lambda_{1}\bm{B}_{k}^{\ast}\bm{u}
=
\bm{B}_{k}^{\ast}\bm{u}
=
\bm{x},
\end{equation*}
and thus conclude that the vector $\bm{x}$ satisfies inequality \eqref{I-xleqlambda1Ax}.
\end{proof}

\subsection{Backtracking procedure for generating solution sets}

Note that, although the generation of the sparsified matrices in $\mathcal{A}$ according to the solution described above is quite a simple task, the number of the matrices in practical problems may be excessively large. Below, we propose a backtracking procedure that allows to reduce the number of matrices under examination. 

The procedure successively checks rows $i=1,\ldots,n$ of the matrix $\bm{A}$ to find and fix one non-zero entry $a_{ij}$ for $j=1,\ldots,n$, and to set the other entries to zero. On selection of an entry in a row, we examine the remaining rows to modify their non-zero entries by setting to $\mathbb{0}$, provided that these entries do not affect the current solution. One step of the procedure is completed when a non-zero entry is fixed in the last row, and hence a sparsified matrix is fully defined.

To prepare the next step, we take the next non-zero entry in the row, provided that such an entry exists. If there is no non-zero entries left in the row, the procedure has to go back to the previous row. It cancels the last selection of non-zero entry, and rolls back the modifications made to the matrix in accordance with the selection. Then, the procedure fixes the next non-zero entry in this row if it exists, or continues back to the previous rows until a new unexplored non-zero entry is found, otherwise. If the new entry is fixed in a row, the procedure continues forward to fix non-zero entries in the next rows, and to modify the remaining rows. The procedure is completed when no more non-zero entries can be selected in the first row. We represent the procedure in recursive form in Algorithm~\ref{A-GenerateSparseMatrices}. 

To describe the modification routine implemented in the procedure in more detail, assume that there are non-zero entries fixed in rows $i=1,\ldots,p-1$, and we now select the entry $a_{pq}$ in row $p$. Since this selection implies that $\lambda_{1}^{-1}a_{pq}x_{q}$ is considered the maximum term in the right-hand side of inequality \eqref{I-xp}, it follows from \eqref{I-apqxq} that
\begin{equation*}
x_{q}\geq a_{pq}^{-1}(a_{pp}\oplus\lambda_{1})x_{p},
\qquad
x_{q}\geq a_{pq}^{-1}a_{pj}x_{j},
\quad
j\ne p.
\end{equation*}

Let us examine the inequality $x_{i}\leq\lambda_{1}^{-1}a_{i1}x_{1}\oplus\cdots\oplus\lambda_{1}^{-1}a_{in}x_{n}$ for $i=p+1,\ldots,n$. If the condition $\lambda_{1}^{-1}a_{iq}a_{pq}^{-1}a_{pi}\geq\mathbb{1}$ holds, then the inequality is fulfilled at the expense of its $q$th term alone, because $\lambda_{1}^{-1}a_{iq}x_{q}\geq\lambda_{1}^{-1}a_{iq}a_{pq}^{-1}a_{pi}x_{i}\geq x_{i}$. Since, in this case, the contribution of the other terms is of no concern, we can set the entries $a_{ij}$ for all $j\ne q$ to $\mathbb{0}$ without changing the solution set under construction.

Suppose that the above condition is not satisfied. Then, we can verify the conditions $a_{iq}a_{pq}^{-1}a_{pj}\geq a_{ij}$ for all $j\ne p,q$, and $a_{iq}a_{pq}^{-1}(a_{pp}\oplus\lambda_{1})\geq a_{ip}$ for $j=p$. If these conditions are satisfied for some $j\ne p$ or $j=p$, then we have $a_{iq}x_{q}\geq a_{iq}a_{pq}^{-1}a_{pj}x_{j}\geq a_{ij}x_{j}$ or $a_{iq}x_{q}\geq a_{iq}a_{pq}^{-1}(a_{pp}\oplus\lambda_{1})x_{p}\geq a_{ip}x_{p}$. This means that term $q$ dominates over term $j$. As before, considering that the last term does not affect the right-hand side of the inequality, we set $a_{ij}=\mathbb{0}$.  

\newcounter{algorithm}
\renewcommand{\thepseudocode}{\arabic{pseudocode}}
\renewcommand{\COMMENT}[1]{\mbox{\bfseries comment:}\ \mbox{#1}}
\begin{pseudocode}{GenerateSparseMatrices}{\bm{A},\mathcal{A}}\label{A-GenerateSparseMatrices}
\PROCEDURE{Backtrack}{\bm{A},p,q}
\COMMENT{Sparsify rows $i\geq p$ in the matrix $\bm{A}=(a_{ij})$}
\\
\IF p\leq n \THEN
\BEGIN
	\COMMENT{Verify whether $a_{pq}$ can be fixed in row $p$}
	\\
	\IF a_{pq}\ne\mathbb{0} \THEN
	\BEGIN
		\COMMENT{Copy $\bm{A}$ into the matrix $\bm{A}^{\prime}=(a_{ij}^{\prime})$}
		\\
		\bm{A}^{\prime} \GETS \bm{A}
		\\
		\COMMENT{Sparsify row $p$ in $\bm{A}^{\prime}$ with $a_{pq}^{\prime}$ fixed}
		\\
		\FOREACH j\ne q \DO
			a_{pj}^{\prime} \GETS \mathbb{0}
			\\	
			\COMMENT{Sparsify rows $i>p$ in $\bm{A}^{\prime}$}
			\\
			\FOR i \GETS p+1 \TO n \DO
			\BEGIN
				\IF \lambda_{1}^{-1}a_{iq}a_{pq}^{-1}a_{pi}\geq\mathbb{1} \THEN
				\BEGIN
					\FOREACH j\ne q \DO
						a_{ij}^{\prime} \GETS \mathbb{0}
				\END
				\ELSE
				\BEGIN
					\FOREACH j\ne q,p \DO
						\BEGIN
							\IF a_{iq}a_{pq}^{-1}a_{pj}\geq a_{ij} \THEN
							a_{ij}^{\prime} \GETS \mathbb{0}
						\END
						\\
						\IF a_{iq}a_{pq}^{-1}(a_{pp}\oplus\lambda_{1})\geq a_{ip} \THEN
						a_{ip}^{\prime}\gets\mathbb{0}
				\END
			\END
			\\
			\IF p=n \THEN
			\BEGIN
				\COMMENT{Store $\bm{A}^{\prime}$ if completed}
				\\
				\mathcal{A} \GETS \mathcal{A}\cup\{\bm{A}^{\prime}\}
			\END
			\ELSE
			\BEGIN
				\COMMENT{Apply recursion otherwise}
				\\
				\FOR j \GETS 1 \TO n \DO
				\CALL{Backtrack}{\bm{A}^{\prime},p+1,j} 
			\END
		\END
	\END
	\ELSE \RETURN{}
\ENDPROCEDURE
\MAIN
\COMMENT{Generate the set $\mathcal{A}$ of sparse matrices from the matrix $\bm{A}$}
\\
\GLOBAL{n,\lambda_{1},\mathcal{A}=\emptyset}
\\
\FOR j \GETS 1 \TO n \DO
	\CALL{Backtrack}{\bm{A},1,j}
\ENDMAIN	
\end{pseudocode}

\subsection{Closed-form representation of complete solution}

We conclude with the representation of the solution to problem \eqref{P-minAxx} in a compact closed form.

\begin{theorem}
\label{T-minAxx-CFR}
Let $\bm{A}$ be a matrix in the refined block-triangular normal form \eqref{E-MNFr}, where the diagonal block $\bm{A}_{11}$ has eigenvalue $\lambda_{1}>\mathbb{0}$.

Denote by $\mathcal{A}$ the set of matrices $\bm{A}_{k}$ that are obtained from $\bm{A}$ by fixing one non-zero entry in each row and by setting the others to $\mathbb{0}$, and that satisfy the condition $\mathop\mathrm{Tr}(\bm{B}_{k})\leq\mathbb{1}$, where $\bm{B}_{k}=\bm{A}_{k}^{-}(\bm{A}\oplus\lambda_{1}\bm{I})$.

Let $\bm{S}$ be the matrix, which is constituted by the maximal linear independent system of columns in the matrices $\bm{B}_{k}^{\ast}$ for all $\bm{A}_{k}\in\mathcal{A}$. 

Then, the minimum value in problem \eqref{P-minAxx} is equal to $\lambda_{1}^{-1}$, and all regular solutions are given by
$$
\bm{x}
=
\bm{S}\bm{v},
\qquad
\bm{v}
>
\bm{0}.
$$
\end{theorem}
\begin{proof}
By Theorem~\ref{T-minAxx}, we find the minimum in the problem to be $\lambda_{1}^{-1}$, and characterize all solutions by the inequality $\bm{x}\leq\lambda_{1}^{-1}\bm{A}\bm{x}$.

According to Theorem~\ref{T-xlambda1Ax}, we define the set $\mathcal{A}$ of matrices $\bm{A}_{k}$ that are obtained from $\bm{A}$ by leaving one of non-zero entries in each row, and such that $\mathop\mathrm{Tr}(\bm{B}_{k})\leq\mathbb{1}$, where $\bm{B}_{k}=\bm{A}_{k}^{-}(\bm{A}\oplus\lambda_{1}\bm{I})$. The theorem provides a family of solutions $\bm{x}=\bm{B}_{k}^{\ast}\bm{u}$, where $\bm{u}>\bm{0}$, for all $\bm{A}_{k}\in\mathcal{A}$. 

Considering that each solution $\bm{x}=\bm{B}_{k}^{\ast}\bm{u}$ defines a subset of vectors generated by the columns of the matrix $\bm{B}_{k}^{\ast}$, we apply Corollary~\ref{C-CuVASM} to represent all solutions as the linear span of the columns in the matrices $\bm{B}_{k}^{\ast}$, corresponding to all $\bm{A}_{k}\in\mathcal{A}$. 

Finally, we reduce the set of all columns by eliminating those, which are linearly dependent on others. We take the remaining columns to form a matrix $\bm{S}$, and then write the solution as $\bm{x}=\bm{S}\bm{v}$, where $\bm{v}>\bm{0}$. 
\end{proof}

\begin{example}
\label{X-imatrix-CFS}
We now apply the results offered by Theorem~\ref{T-minAxx-CFR} to derive all solutions of the problem considered in Example~\ref{X-rmatrix}. We take the matrix $\bm{A}$ given by \eqref{E-A100320001}, and replace one non-zero entry in the second and third rows of $\bm{A}$ by $\mathbb{0}=-\infty$ to produce the sparsified matrices
\begin{gather*}
\bm{A}_{1}
=
\left(
\begin{array}{ccr}
1 & \mathbb{0} & \mathbb{0}
\\
3 & \mathbb{0} & \mathbb{0}
\\
\mathbb{0} & 0 & \mathbb{0}
\end{array}
\right),
\qquad
\bm{A}_{2}
=
\left(
\begin{array}{ccr}
1 & \mathbb{0} & \mathbb{0}
\\
\mathbb{0} & 2 & \mathbb{0}
\\
\mathbb{0} & 0 & \mathbb{0}
\end{array}
\right),
\\
\bm{A}_{3}
=
\left(
\begin{array}{ccr}
1 & \mathbb{0} & \mathbb{0}
\\
3 & \mathbb{0} & \mathbb{0}
\\
\mathbb{0} & \mathbb{0} & -1
\end{array}
\right),
\qquad
\bm{A}_{4}
=
\left(
\begin{array}{ccr}
1 & \mathbb{0} & \mathbb{0}
\\
\mathbb{0} & 2 & \mathbb{0}
\\
\mathbb{0} & \mathbb{0} & -1
\end{array}
\right).
\end{gather*}

Next, taking into account that $\lambda_{1}=1$, we calculate the matrices
\begin{gather*}
\bm{B}_{1}
=
\bm{A}_{1}^{-}(\bm{A}\oplus\lambda_{1}\bm{I})
=
\left(
\begin{array}{crc}
0 & -1 & \mathbb{0}
\\
\mathbb{0} & 0 & 1
\\
\mathbb{0} & \mathbb{0} & \mathbb{0}
\end{array}
\right),
\\
\bm{B}_{2}
=
\bm{A}_{2}^{-}(\bm{A}\oplus\lambda_{1}\bm{I})
=
\left(
\begin{array}{ccc}
0 & \mathbb{0} & \mathbb{0}
\\
1 & 0 & 1
\\
\mathbb{0} & \mathbb{0} & \mathbb{0}
\end{array}
\right),
\\
\bm{B}_{3}
=
\bm{A}_{3}^{-}(\bm{A}\oplus\lambda_{1}\bm{I})
=
\left(
\begin{array}{crc}
0 & -1 & \mathbb{0}
\\
\mathbb{0} & \mathbb{0} & \mathbb{0}
\\
\mathbb{0} & 1 & 2
\end{array}
\right),
\\
\bm{B}_{4}
=
\bm{A}_{4}^{-}(\bm{A}\oplus\lambda_{1}\bm{I})
=
\left(
\begin{array}{ccc}
0 & \mathbb{0} & \mathbb{0}
\\
1 & 0 & \mathbb{0}
\\
\mathbb{0} & 1 & 2
\end{array}
\right).
\end{gather*}

Furthermore, we obtain $\mathop\mathrm{Tr}(\bm{B}_{1})=\mathop\mathrm{Tr}(\bm{B}_{2})=0=\mathbb{1}$. Since the matrices $\bm{B}_{1}$ and $\bm{B}_{2}$ satisfy the condition of the theorem, they are accepted. Considering that $\mathop\mathrm{Tr}(\bm{B}_{3})=\mathop\mathrm{Tr}(\bm{B}_{4})=2>\mathbb{1}$, the last two matrices are rejected.

To represent all solutions of the problem, we need to calculate the matrices
\begin{gather*}
\bm{B}_{1}^{\ast}
=
\bm{I}\oplus\bm{B}_{1}\oplus\bm{B}_{1}^{2}
=
\left(
\begin{array}{crc}
0 & -1 & 0
\\
\mathbb{0} & 0 & 1
\\
\mathbb{0} & \mathbb{0} & 0
\end{array}
\right),
\\
\bm{B}_{2}^{\ast}
=
\bm{I}\oplus\bm{B}_{2}\oplus\bm{B}_{2}^{2}
=
\left(
\begin{array}{ccc}
0 & \mathbb{0} & \mathbb{0}
\\
1 & 0 & 1
\\
\mathbb{0} & \mathbb{0} & 0
\end{array}
\right),
\end{gather*}
and then examine the set of their columns.

Observing that the last two columns in the matrix $\bm{B}_{1}^{\ast}$ and the first in $\bm{B}_{2}^{\ast}$ are linear combinations of others, since
\begin{gather*}
\left(
\begin{array}{r}
-1
\\
0
\\
\mathbb{0}
\end{array}
\right)
=
(-1)
\left(
\begin{array}{c}
0
\\
1
\\
\mathbb{0}
\end{array}
\right)
=
(-1)
\left(
\begin{array}{c}
0
\\
\mathbb{0}
\\
\mathbb{0}
\end{array}
\right)
\oplus
\left(
\begin{array}{c}
\mathbb{0}
\\
0
\\
\mathbb{0}
\end{array}
\right),
\\
\left(
\begin{array}{c}
0
\\
1
\\
0
\end{array}
\right)
=
\left(
\begin{array}{c}
0
\\
\mathbb{0}
\\
\mathbb{0}
\end{array}
\right)
\oplus
\left(
\begin{array}{c}
\mathbb{0}
\\
1
\\
0
\end{array}
\right),
\end{gather*}
we can drop these columns without loss of solutions. With the matrix $\bm{S}$, which is formed by the independent columns, and the notation $\bm{v}=(v_{1},v_{2},v_{3})^{T}$, a complete solution to the problem is represented as
\begin{equation*}
\bm{x}
=
\bm{S}\bm{v},
\qquad
\bm{S}
=
\left(
\begin{array}{ccc}
0 & \mathbb{0} & \mathbb{0}
\\
\mathbb{0} & 0 & 1
\\
\mathbb{0} & \mathbb{0} & 0
\end{array}
\right),
\qquad
\bm{v}
\in
\mathbb{R}^{3}.
\end{equation*}

Note that, with $v_{3}=(-1)v_{2}$ (or, in the usual notation, $v_{3}=v_{2}-1$), the solution reduces to the one obtained in Example~\ref{X-rmatrix}. 
\qed
\end{example}

\section{Optimization problem with composite objective function}
\label{OPComOF}

In this section, we extend the solutions obtained for problem \eqref{P-minAxx} with component objective function to problem \eqref{P-minxAxAxx} with composite function.

\begin{theorem}
\label{T-minxAxAxx}
Let $\bm{A}$ be a matrix in the refined block-triangular normal form \eqref{E-MNFr}, where the diagonal block $\bm{A}_{11}$ has eigenvalue $\lambda_{1}>\mathbb{0}$.

Then, the minimum value in problem \eqref{P-minxAxAxx} is equal to
\begin{equation*}
\mu
=
\lambda_{1}\oplus\cdots\oplus\lambda_{s}\oplus\lambda_{1}^{-1},
\end{equation*}
and all regular solutions are characterized by the system of inequalities
\begin{equation}
\bm{A}\bm{x}
\leq
\mu\bm{x},
\qquad
\bm{x}
\leq
\mu\bm{A}\bm{x}.
\label{I-Axleqmux-xleqmuAx}
\end{equation}

Specifically, any block vector $\bm{x}^{T}=(\bm{x}_{1}^{T},\ldots,\bm{x}_{s}^{T})$ with the blocks $\bm{x}_{i}$ defined successively for each $i=1,\ldots,s$ by the conditions 
\begin{equation*}
\bm{x}_{i}
=
\begin{cases}
(\lambda_{i}^{-1}\bm{A}_{ii})^{\times}\bm{u}_{i},
&
\text{if $i\leq r$};
\\
\lambda_{i}^{-1}(\lambda_{i}^{-1}\bm{A}_{ii})^{\ast}\displaystyle\bigoplus_{j=1}^{i-1}\bm{A}_{ij}\bm{x}_{j}\oplus(\lambda_{i}^{-1}\bm{A}_{ii})^{\times}\bm{u}_{i},
&
\text{if $i>r$ and $\lambda_{i}\geq\lambda_{1}$};
\\
\lambda_{1}^{-1}(\lambda_{1}^{-1}\bm{A}_{ii})^{\ast}\displaystyle\bigoplus_{j=1}^{i-1}\bm{A}_{ij}\bm{x}_{j},
&
\text{if $i>r$ and $\lambda_{i}<\lambda_{1}$};
\end{cases}
\end{equation*}
where $\bm{u}_{i}$ are regular vectors of appropriate size, is a solution of the problem.
\end{theorem}

\begin{proof}
With the spectral radius of the matrix $\bm{A}$ denoted by $\lambda=\lambda_{1}\oplus\cdots\oplus\lambda_{s}$, we apply Lemma~\ref{L-minxAx} and Theorem~\ref{T-minAxx} to obtain the lower bound
\begin{equation*} 
\bm{x}^{-}\bm{A}\bm{x}
\oplus
(\bm{A}\bm{x})^{-}\bm{x}
\geq
\lambda\oplus\lambda_{1}^{-1}
=
\mu.
\end{equation*} 

To verify that $\mu$ is the strict lower bound of the objective function, and thus the minimum in \eqref{P-minxAxAxx}, we now find a vector $\bm{x}$ that attains this bound. We have to solve the inequality $\bm{x}^{-}\bm{A}\bm{x}\oplus(\bm{A}\bm{x})^{-}\bm{x}\leq\mu$, which is equivalent to the two inequalities
\begin{equation}
\bm{x}^{-}\bm{A}\bm{x}
\leq
\mu,
\qquad
(\bm{A}\bm{x})^{-}\bm{x}
\leq
\mu.
\label{I-xAxleqmu-Axxleqmu}
\end{equation}

Due to the refined block-triangular form of the matrix $\bm{A}$, we can rewrite the last inequalities as the system
\begin{gather}
\bm{x}_{i}^{-}\bm{A}_{ii}\bm{x}_{i}
\leq
\mu,
\qquad
(\bm{A}_{ii}\bm{x}_{i})^{-}\bm{x}_{i}
\leq
\mu,
\qquad
i\leq r;
\label{I-xiAiixileqmu-Aiixixileqmu}
\\
\bm{x}_{i}^{-}\bigoplus_{j=1}^{i}\bm{A}_{ij}\bm{x}_{j}
\leq
\mu,
\qquad
\left(\bigoplus_{j=1}^{i}\bm{A}_{ij}\bm{x}_{j}\right)^{-}\bm{x}_{i}
\leq
\mu,
\qquad
i>r.
\label{I-xiAijxjleqmu-Aijxjxileqmu}
\end{gather}

To find vectors $\bm{x}_{i}$ for $i=1,\ldots,s$, we first assume that $i\leq r$, and consider inequalities \eqref{I-xiAiixileqmu-Aiixixileqmu}. We assume $\bm{x}_{i}$ to be an eigenvector of $\bm{A}_{ii}$, given by the equation $\bm{A}_{ii}\bm{x}_{i}=\lambda_{i}\bm{x}_{i}$, which is solved as
\begin{equation*}
\bm{x}_{i}
=
(\lambda_{i}^{-1}\bm{A}_{ii})^{\times}\bm{u}_{i},
\qquad
\bm{u}_{i}
>
\bm{0}.
\end{equation*}

Furthermore, since $\lambda_{i}^{-1}\leq\lambda_{1}^{-1}$ for $i\leq r$, we have
\begin{equation*}
\bm{x}_{i}^{-}\bm{A}_{ii}\bm{x}_{i}
=
\lambda_{i}
\leq
\lambda
\leq
\mu,
\qquad
(\bm{A}_{ii}\bm{x}_{i})^{-}\bm{x}_{i}
=
\lambda_{i}^{-1}
\leq
\lambda_{1}^{-1}
\leq
\mu,
\end{equation*}
which means that inequalities \eqref{I-xiAiixileqmu-Aiixixileqmu} are fulfilled.

We consider inequalities \eqref{I-xiAijxjleqmu-Aijxjxileqmu} and examine two cases. If $\lambda_{i}\geq\lambda_{1}$, we take a vector $\bm{x}_{i}$ to satisfy the equation $\bm{A}_{i1}\bm{x}_{1}\oplus\cdots\oplus\bm{A}_{ii}\bm{x}_{i}=\lambda_{i}\bm{x}_{i}$. By Theorem~\ref{T-Axbeqx}, the equation has the solution
\begin{equation*}
\bm{x}_{i}
=
\lambda_{i}^{-1}
(\lambda_{i}^{-1}\bm{A}_{ii})^{\ast}
\bigoplus_{j=1}^{i-1}\bm{A}_{ij}\bm{x}_{j}
\oplus
(\lambda_{i}^{-1}\bm{A}_{ii})^{\times}\bm{u}_{i},
\qquad
\bm{u}_{i}
>
\bm{0}.
\end{equation*}

In this case, inequalities \eqref{I-xiAijxjleqmu-Aijxjxileqmu} hold since
\begin{equation*}
\bm{x}_{i}^{-}\bigoplus_{j=1}^{i}\bm{A}_{ij}\bm{x}_{j}
=
\lambda_{i}
\leq
\lambda
\leq
\mu,
\qquad
\left(\bigoplus_{j=1}^{i}\bm{A}_{ij}\bm{x}_{j}\right)^{-}\bm{x}_{i}
=
\lambda_{i}^{-1}
\leq
\lambda_{1}^{-1}
\leq
\mu.
\end{equation*}

Under the condition $\lambda_{i}<\lambda_{1}$, we define $\bm{x}_{i}$ as a solution of the equation $\lambda_{1}^{-1}(\bm{A}_{i1}\bm{x}_{1}\oplus\cdots\oplus\bm{A}_{ii}\bm{x}_{i})=\bm{x}_{i}$. Application of  Theorem~\ref{T-Axbeqx} leads to the result
\begin{equation*}
\bm{x}_{i}
=
\lambda_{i}^{-1}
(\lambda_{i}^{-1}\bm{A}_{ii})^{\ast}
\bigoplus_{j=1}^{i-1}\bm{A}_{ij}\bm{x}_{j}.
\end{equation*}

Substitution of the solution into inequalities \eqref{I-xiAijxjleqmu-Aijxjxileqmu} yields
\begin{equation*}
\bm{x}_{i}^{-}\bigoplus_{j=1}^{i}\bm{A}_{ij}\bm{x}_{j}
=
\lambda_{1}
\leq
\mu,
\qquad
\left(\bigoplus_{j=1}^{i}\bm{A}_{ij}\bm{x}_{j}\right)^{-}\bm{x}_{i}
=
\lambda_{1}^{-1}
\leq
\mu.
\end{equation*}
which means that these inequalities are also fulfilled.

It remains to apply Lemma~\ref{L-Axleqd} to inequalities \eqref{I-xAxleqmu-Axxleqmu} to obtain \eqref{I-Axleqmux-xleqmuAx}.
\end{proof}

\begin{example}
Consider problem \eqref{P-minxAxAxx}, where the matrix $\bm{A}$ is defined as \eqref{E-A100320001}. To apply the solution offered by Theorem~\ref{T-minxAxAxx}, we use some preliminary results from Example~\ref{X-rmatrix}.

First, we obtain the minimum value $\mu=\lambda_{1}\oplus\lambda_{2}\oplus\lambda_{3}\oplus\lambda_{1}^{-1}=2$, and define $x_{1}=u_{1}$, where $u_{1}\in\mathbb{R}$. Furthermore, observing that $\lambda_{2}>\lambda_{1}$, we write $x_{2}=\lambda_{2}^{-1}(\lambda_{2}^{-1}\bm{A}_{22})^{\ast}\bm{A}_{21}x_{1}\oplus(\lambda_{2}^{-1}\bm{A}_{22})^{\times}u_{2}=1x_{1}\oplus u_{2}=1u_{1}\oplus u_{2}$, or, in the usual notation, $x_{2}=\max(u_{1}+1,u_{2})$.

Finally, in a similar way as in Example~\ref{X-rmatrix}, we have $x_{3}=(-1)x_{2}=(-1)(1u_{1}\oplus u_{2})=u_{1}\oplus(-1)u_{2}$ (or $x_{3}=\max(u_{1},u_{2}-1)$).

Combining the results, we represent the solution in vector form as
\begin{equation*}
\bm{x}
=
\left(
\begin{array}{cr}
0 & \mathbb{0}
\\
1 & 0
\\
0 & -1
\end{array}
\right)
\bm{u},
\qquad
\bm{u}\in\mathbb{R}^{2}.
\qed
\end{equation*}
\end{example}

Next, we consider a special case of the problem, for which a complete solution can be derived as a consequence of the previous result.
\begin{corollary}
Let $\bm{A}$ be a matrix in the refined block-triangular normal form \eqref{E-MNFr}, where the diagonal blocks $\bm{A}_{ii}$ have eigenvalues $\lambda_{i}=\mathbb{1}$ for all $i=1,\ldots,s$.

Then, the minimum value in problem \eqref{P-minxAxAxx} is equal to $\mathbb{1}$, and all regular solutions are given by the condition
\begin{equation*}
\bm{x}_{i}
=
\begin{cases}
\bm{A}_{ii}^{\times}\bm{u}_{i},
&
\text{if $i\leq r$};
\\
\bm{A}_{ii}^{\ast}\displaystyle\bigoplus_{j=1}^{i-1}\bm{A}_{ij}\bm{x}_{j}\oplus\bm{A}_{ii}^{\times}\bm{u}_{i},
&
\text{if $i>r$};
\end{cases}
\end{equation*}
where $\bm{u}_{i}$ are regular vectors of appropriate size.
\end{corollary}

\begin{proof}
With $\mu=\mathbb{1}$, we represent inequalities \eqref{I-Axleqmux-xleqmuAx} as the double inequality $\bm{A}\bm{x}\leq\bm{x}\leq\bm{A}\bm{x}$, which is equivalent to the equality $\bm{A}\bm{x}=\bm{x}$. After representation of the vector $\bm{x}$ in block form, and successive application of Theorem~\ref{T-Axbeqx} to the equations for each block, we obtain the desired result. 
\end{proof}

In a similar way as for problem \eqref{P-minxAxAxx}, it is easy to verify the next statement.
\begin{corollary}
Under the conditions of Theorem~\ref{T-minxAxAxx}, the set of solution vectors of problem \eqref{P-minxAxAxx} is closed under vector addition and scalar multiplication.
\end{corollary}

\subsection{Derivation of complete solution}

We now follow similar arguments as before to describe a complete solution for the problem with the composite objective function under consideration.  

\begin{theorem}
Let $\bm{A}$ be a matrix in the refined block-triangular normal form \eqref{E-MNFr}, where the block $\bm{A}_{11}$ has eigenvalue $\lambda_{1}>\mathbb{0}$, and $\mu=\lambda_{1}\oplus\cdots\oplus\lambda_{s}\oplus\lambda_{1}^{-1}$. 

Denote by $\mathcal{A}$ the set of matrices $\bm{A}_{k}$ that are obtained from $\bm{A}$ by fixing one non-zero entry in each row and by setting the others to $\mathbb{0}$, and that satisfy the condition $\mathop\mathrm{Tr}(\bm{B}_{k})\leq\mathbb{1}$, where $\bm{B}_{k}=\bm{A}_{k}^{-}\bm{A}\oplus\mu^{-1}(\bm{A}_{k}^{-}\oplus\bm{A})$.

Then, all regular solutions of the system of inequalities \eqref{I-Axleqmux-xleqmuAx} are given by the conditions
\begin{equation}
\bm{x}
=
\bm{B}_{k}^{\ast}\bm{u},
\qquad
\bm{B}_{k}
=
\bm{A}_{k}^{-}\bm{A}\oplus\mu^{-1}(\bm{A}_{k}^{-}\oplus\bm{A}),
\qquad
\bm{A}_{k}
\in
\mathcal{A},
\qquad
\bm{u}
>
\bm{0}.
\label{E-xBku-BkAkAmu1AkA-AkA-u0}
\end{equation}
\end{theorem}
\begin{proof}
In the similar way as for the problem with component objective function, we see that regular solutions to the system at \eqref{I-Axleqmux-xleqmuAx} exists.

To prove the theorem, we need to verify that any regular solution of the system at \eqref{I-Axleqmux-xleqmuAx} can be represented as \eqref{E-xBku-BkAkAmu1AkA-AkA-u0}, and vice versa.

Suppose that $\bm{x}$ is a regular solution of system \eqref{I-Axleqmux-xleqmuAx}. Observing that the vector $\bm{x}$ satisfies the second inequality $\bm{x}\leq\mu\bm{A}\bm{x}$ in \eqref{I-Axleqmux-xleqmuAx}, we use the same arguments as in the proof of Theorem~\ref{T-xlambda1Ax} to conclude that $\bm{x}$ satisfies the inequality $\bm{x}\geq\bm{A}_{k}^{-}(\bm{A}\oplus\mu^{-1}\bm{I})\bm{x}$, where $\bm{A}_{k}$ is a matrix obtained from $\bm{A}$ by leaving only one of non-zero entries in each row. Combining with the first inequality in the form $\bm{x}\geq\mu^{-1}\bm{A}\bm{x}$ results in
\begin{equation*}
\bm{x}
\geq
(\bm{A}_{k}^{-}\bm{A}\oplus\mu^{-1}(\bm{A}_{k}^{-}\oplus\bm{A}))\bm{x}
=
\bm{B}_{k}\bm{x}.
\end{equation*}

This inequality has a regular solution $\bm{x}$, which, by Theorem~\ref{T-Axleqx}, means that $\mathop\mathrm{Tr}(\bm{B}_{k})\leq\mathbb{1}$. In this case, we have $\bm{x}=\bm{B}_{k}^{\ast}\bm{u}$ for some $\bm{u}>\bm{0}$, and hence represent $\bm{x}$ in the form of \eqref{E-xBku-BkAkAmu1AkA-AkA-u0}.

Consider any vector $\bm{x}$ defined by the condition at \eqref{E-xBku-BkAkAmu1AkA-AkA-u0}, and verify that this vector satisfies both inequalities at \eqref{I-Axleqmux-xleqmuAx}. First, we note that $\mu\bm{B}_{k}=\mu\bm{A}_{k}^{-}\bm{A}\oplus\bm{A}_{k}^{-}\oplus\bm{A}\geq\bm{A}_{k}^{-}\oplus\bm{A}\geq\bm{A}$. Moreover, since $\mathop\mathrm{Tr}(\bm{B}_{k})\leq\mathbb{1}$, we have $\bm{B}_{k}^{+}=\bm{B}_{k}\bm{B}_{k}^{\ast}\leq\bm{B}_{k}^{\ast}$. As a result we obtain $\bm{A}\bm{B}_{k}^{\ast}\leq\mu\bm{B}_{k}\bm{B}_{k}^{\ast}=\mu\bm{B}_{k}^{+}\leq\mu\bm{B}_{k}^{\ast}$.

Therefore, we have $\bm{A}\bm{x}=\bm{A}\bm{B}_{k}^{\ast}\bm{u}\leq\mu\bm{B}_{k}^{\ast}\bm{u}=\mu\bm{x}$, which yields the left inequality at \eqref{I-Axleqmux-xleqmuAx}.

Furthermore, considering that $\mu\bm{B}_{k}\geq\bm{A}_{k}^{-}\oplus\bm{A}\geq\bm{A}_{k}^{-}$, we obtain
\begin{equation*}
\mu\bm{A}\bm{B}_{k}^{\ast}
=
\mu\bm{A}\bigoplus_{m=0}^{n-1}\bm{B}_{k}^{m}
\geq
\mu\bm{A}\bigoplus_{m=1}^{n}\bm{B}_{k}^{m}
=
\mu\bm{A}\bm{B}_{k}
\bigoplus_{m=0}^{n-1}\bm{B}_{k}^{m}
\geq
\bm{A}\bm{A}_{k}^{-}\bm{B}_{k}^{\ast}
\geq
\bm{B}_{k}^{\ast}.
\end{equation*}

Finally, we write $\mu\bm{A}\bm{x}=\mu\bm{A}\bm{B}_{k}^{\ast}\bm{u}\geq\bm{B}_{k}^{\ast}\bm{u}=\bm{x}$, and thus the right inequality at \eqref{I-Axleqmux-xleqmuAx} also holds.
\end{proof}

\subsection{Backtracking procedure for generating solution sets}

The backtracking procedure described above for the problem with component objective function to generate solutions of inequality \eqref{I-xleqlambda1Ax} can serve, after replacing $\lambda^{-1}$ by $\mu$, as an appropriate tool for generating solutions for the composite problem under study as well. Moreover, the characterization of solutions by two inequalities \eqref{I-Axleqmux-xleqmuAx} instead of one \eqref{I-xleqlambda1Ax} makes it possible to improve the procedure by reducing the number of non-zero entries in the initial matrix prior to generating sparsified matrices.

Consider the first vector inequality at \eqref{I-Axleqmux-xleqmuAx}, and note that it is equivalent to the system of scalar inequalities $\mu x_{i}\geq a_{ij}x_{j}$ for all $i,j=1,\ldots,n$.

Suppose that the condition $a_{ip}a_{pq}\geq\mu a_{iq}$ is valid for some indices $i$, $p$ and $q$, and examine the scalar inequality $x_{i}\leq\mu a_{i1}x_{1}\oplus\cdots\oplus\mu a_{in}x_{n}$ from the second inequality at \eqref{I-Axleqmux-xleqmuAx}. Since $\mu a_{ip}x_{p}\geq a_{ip}a_{pq}x_{q}\geq\mu a_{iq}x_{q}$, the term $\mu a_{ip}x_{p}$ dominates over $\mu a_{iq}x_{q}$. As a result, the last term can be eliminated by setting $a_{iq}=\mathbb{0}$, which does not alter the solution of the problem.

Finally, under the condition $a_{ip}a_{pi}\geq\mathbb{1}$, we have $\mu a_{ip}x_{p}\geq a_{ip}a_{pi}x_{i}\geq x_{i}$, which means that the term $\mu a_{ip}x_{p}$ makes the scalar inequality hold independently of the other terms. For these insufficient terms, we again replace the matrix entries $a_{ij}$ by $\mathbb{0}$ for all $j\ne p$ without affecting the solutions.

\subsection{Closed-form representation of complete solution}

Theorem~\ref{T-minAxx-CFR} is readily extended to the composite problem as follows.

\begin{theorem}
\label{T-minxAxAxx-CFR}
Let $\bm{A}$ be a matrix in the refined block-triangular normal form \eqref{E-MNFr}, where the block $\bm{A}_{11}$ has eigenvalue $\lambda_{1}>\mathbb{0}$, and $\mu=\lambda_{1}\oplus\cdots\oplus\lambda_{s}\oplus\lambda_{1}^{-1}$.

Denote by $\mathcal{A}$ the set of matrices $\bm{A}_{k}$ that are obtained from $\bm{A}$ by fixing one non-zero entry in each row and by setting the others to $\mathbb{0}$, and that satisfy the condition $\mathop\mathrm{Tr}(\bm{B}_{k})\leq\mathbb{1}$, where $\bm{B}_{k}=\bm{A}_{k}^{-}\bm{A}\oplus\mu^{-1}(\bm{A}_{k}^{-}\oplus\bm{A})$.

Let $\bm{S}$ be the matrix, which is constituted by the maximal linear independent system of columns in the matrices $\bm{B}_{k}^{\ast}$ for all $\bm{A}_{k}\in\mathcal{A}$. 

Then, the minimum value in problem \eqref{P-minAxx} is equal to $\mu$, and all regular solutions are given by
$$
\bm{x}
=
\bm{S}\bm{v},
\qquad
\bm{v}
>
\bm{0}.
$$
\end{theorem}

\begin{example}
Let us apply Theorem~\ref{T-minxAxAxx-CFR} to derive all solutions of problem \eqref{P-minxAxAxx} with the matrix defined by \eqref{E-A100320001}. We take the minimum value $\mu=2$, and the sparsified matrices $\bm{A}_{1}$, $\bm{A}_{2}$, $\bm{A}_{3}$ and $\bm{A}_{4}$, obtained in Example~\ref{X-imatrix-CFS}, to calculate the matrices
\begin{gather*}
\bm{B}_{1}
=
\bm{A}_{1}^{-}\bm{A}\oplus\mu^{-1}(\bm{A}_{1}^{-}\oplus\bm{A})
=
\left(
\begin{array}{crr}
0 & -1 & \mathbb{0}
\\
1 & 0 & -1
\\
\mathbb{0} & -2 & -3
\end{array}
\right),
\\
\bm{B}_{2}
=
\bm{A}_{2}^{-}\bm{A}\oplus\mu^{-1}(\bm{A}_{2}^{-}\oplus\bm{A})
=
\left(
\begin{array}{crr}
0 & \mathbb{0} & \mathbb{0}
\\
1 & 0 & -1
\\
\mathbb{0} & -2 & -3
\end{array}
\right),
\\
\bm{B}_{3}
=
\bm{A}_{3}^{-}\bm{A}\oplus\mu^{-1}(\bm{A}_{3}^{-}\oplus\bm{A})
=
\left(
\begin{array}{crc}
0 & -1 & \mathbb{0}
\\
1 & 0 & \mathbb{0}
\\
\mathbb{0} & 1 & 0
\end{array}
\right),
\\
\bm{B}_{4}
=
\bm{A}_{4}^{-}\bm{A}\oplus\mu^{-1}(\bm{A}_{4}^{-}\oplus\bm{A})
=
\left(
\begin{array}{ccc}
0 & \mathbb{0} & \mathbb{0}
\\
1 & 0 & \mathbb{0}
\\
\mathbb{0} & 1 & 0
\end{array}
\right).
\end{gather*}

Since $\mathop\mathrm{Tr}(\bm{B}_{1})=\mathop\mathrm{Tr}(\bm{B}_{2})=\mathop\mathrm{Tr}(\bm{B}_{3})=\mathop\mathrm{Tr}(\bm{B}_{4})=0=\mathbb{1}$, all matrices satisfy the condition of Theorem~\ref{T-minxAxAxx-CFR}, and thus are accepted.

Furthermore, we obtain the matrices
\begin{gather*}
\bm{B}_{1}^{\ast}
=
\bm{I}\oplus\bm{B}_{1}\oplus\bm{B}_{1}^{2}
=
\left(
\begin{array}{rrr}
0 & -1 & -2
\\
1 & 0 & -1
\\
-1 & -2 & 0
\end{array}
\right),
\\
\bm{B}_{2}^{\ast}
=
\bm{I}\oplus\bm{B}_{2}\oplus\bm{B}_{2}^{2}
=
\left(
\begin{array}{rrr}
0 & \mathbb{0} & \mathbb{0}
\\
1 & 0 & -1
\\
-1 & -2 & 0
\end{array}
\right),
\\
\bm{B}_{3}^{\ast}
=
\bm{I}\oplus\bm{B}_{3}\oplus\bm{B}_{3}^{2}
=
\left(
\begin{array}{crc}
0 & -1 & \mathbb{0}
\\
1 & 0 & \mathbb{0}
\\
2 & 1 & 0
\end{array}
\right),
\\
\bm{B}_{4}^{\ast}
=
\bm{I}\oplus\bm{B}_{4}\oplus\bm{B}_{4}^{2}
=
\left(
\begin{array}{ccc}
0 & \mathbb{0} & \mathbb{0}
\\
1 & 0 & \mathbb{0}
\\
2 & 1 & 0
\end{array}
\right).
\end{gather*}

We now consider the set of columns of these matrices to find and eliminate those columns, which are linearly dependent on others. We take the rest of the columns to form the matrix $\bm{S}$, and represent all solutions as
\begin{equation*}
\bm{x}
=
\bm{S}\bm{v},
\qquad
\bm{S}
=
\left(
\begin{array}{rrc}
0 & \mathbb{0} & \mathbb{0}
\\
1 & 0 & \mathbb{0}
\\
-1 & -2 & 0
\end{array}
\right),
\qquad
\bm{v}
\in
\mathbb{R}^{3}.
\end{equation*}

It remains to note that, under the condition $v_{3}=v_{1}\oplus(-1)v_{2}$ (or, in the usual notation, $v_{3}=\max(v_{1},v_{2}-1$)), the obtained solution reduces to the partial solution given in Example~\ref{X-imatrix-CFS}. 
\qed
\end{example}

\section{Conclusions}
\label{S-C}

The paper focused on the development of methods and techniques for the complete solution of optimization problems, formulated in the framework of tropical mathematics to minimize nonlinear functions defined by a matrix on vectors over idempotent semifield. As starting point, we have taken our previous results, which offer partial solutions to the problems with both irreducible and reducible matrices. To extend these results further, we have derived a characterization of the solutions in the form of vector inequalities. We have developed an approach to describe all solutions of the problem as a family of solution subsets by using a matrix sparsification technique. To generate all members of the family in a reasonable way, we have proposed a backtracking procedure. Finally, we have represented the complete solutions of the problems in a compact vector form, ready for further analysis and calculation. The obtained results were illustrated with numerical examples.

The directions of future research will include the development of real-world applications of the proposed solutions. A detailed analysis of the computational complexity of the backtracking procedure is of particular interest. Various extensions of the solution to handle other classes of optimization problems with different objective functions and constraints, and in different algebraic settings are also considered promising lines of future investigation. 

\section*{Acknowledgments}
This work was supported in part by the Russian Foundation for Basic Research (grant number 18-010-00723). The author is very grateful to the referees for their extremely valuable comments and suggestions, which have been incorporated into the revised version of the manuscript.

\bibliographystyle{abbrvurl}

\bibliography{Complete_algebraic_solution_of_multidimensional_optimization_problems_in_tropical_semifield}

\end{document}